\documentclass[12pt]{amsart}

\usepackage{
amsfonts,
latexsym,
amssymb,
mathabx,
enumerate,
verbatim,
mathrsfs,
}

\newcommand{\labbel}{\label}

\newtheorem{theorem}{Theorem}[section]
\newtheorem{lemma}[theorem]{Lemma}

\newtheorem{proposition}[theorem]{Proposition} 
 
\newtheorem{corollary}[theorem]{Corollary} 
 
\newtheorem*{claim}{Claim}

\newtheorem*{theorem*}{Theorem}
\newtheorem*{corollary*}{Corollary}

\theoremstyle{definition}
\newtheorem{definition}[theorem]{Definition}

\theoremstyle{remark}
\newtheorem{remark}[theorem]{Remark}

\newtheorem*{acknowledgement}{Acknowledgement}

\newcommand{\+}{\mathbin{\hash}}

\DeclareMathOperator*{\nsum}
    {\ifdim\displaywidth>0pt {{\mbox{\raisebox{-0.4ex}{\LARGE$\hash$}}}}
     \else{{\mbox{\raisebox{-0.3ex}{\Large$\hash$}}}}
\fi}

\newcommand{\x}{\mathbin{\otimes}}

\newcommand{\xx}{\bigotimes} 

\newcommand{\prodh}{\prod^{\#}}

\newcommand{\LL}[1]
{{\text{L}^*_{#1}}}

\newcommand{\LLL}[1]
{{\text{L}^0_{#1}}}

\newcommand{\diff}{dif\hspace{-2.5pt} f}

 \changenotsign  

 \allowdisplaybreaks[1]

\begin{document}
 
\title
{An infinite natural product}

\author{Paolo Lipparini} 
\address{ 
Dipartimento Produttivo di Matematica\\Viale della  Ricerca
 Scientifica\\Universit\`a di Roma ``Tor Vergata'' 
\\I-00133 ROME ITALY}
\urladdr{http://www.mat.uniroma2.it/\textasciitilde lipparin}

\keywords{Ordinal number; (infinite) (natural) product, sum; 
(locally) finitely Carruth extension}

\subjclass[2010]{03E10, 06A05}
\thanks{Work performed under the auspices of G.N.S.A.G.A. Work 
partially supported by PRIN 2012 ``Logica, Modelli e Insiemi''}

\begin{abstract}
We study an infinite countable iteration of the natural product
between ordinals. We present an ``effective''   way to compute 
this countable natural product; in the non trivial cases
the result depends only on the natural sum of the degrees of the 
factors, where the degree of a nonzero ordinal is the largest exponent in its
Cantor normal form representation. Thus we are able to lift
former results about infinitary sums to infinitary products.
Finally, we provide an order-theoretical characterization
of the infinite natural product; this characterization merges
in a nontrivial way 
a theorem by Carruth describing the natural product of two ordinals and
a known description of the ordinal product of a possibly
infinite sequence of ordinals. 
\end{abstract} 

\maketitle

\section{Introduction} \labbel{intro} 

The  usual addition $+$ and multiplication 
$\cdot$ between 
ordinals can be defined by transfinite recursion and have
a clear order-theoretical meaning; for example, $\alpha+ \beta $
is the order-type of a copy of $\alpha$ to which a copy 
of $\beta$ is added at the top.
However,  $+$ and  $\cdot$
have very poor algebraic properties; though both are associative, 
 they are neither commutative nor cancellative,
left distributivity fails, etc.
See, e.~g., Bachmann  \cite{Bac},
Hausdorff \cite{Ha}  
and Sierpi{\'n}ski  \cite{Sier} for full details.

In certain cases it is useful to consider the so-called \emph{natural}
operations $\+$ and $\x$.   
These operations are defined by expressing 
the operands in
Cantor normal form and, roughly, treating
the expressions as if they were polynomials in $ \omega$.
The natural operations have the advantage of satisfying
good algebraic properties; moreover,  in a few cases they have found 
mathematical applications even outside logic.
See e.~g., Carruth \cite{Car} and  
Toulmin \cite{T}.  
Further references, including some variants and  historical remarks,
can be found in Altman \cite{A} and Ehrlich \cite[pp. 24--25]{E}. 
It is also interesting to observe that the natural 
operations are
the restriction to the ordinals
of the surreal  operations on  Conway ``Numbers'' \cite{Co}.
Limited to the case of the ordinal natural sum, the corresponding 
two-arguments recursive definition appears implicitly as early as 
in the proof of  de Jongh and  Parikh \cite[Theorem 3.4]{dJP}. 

An infinitary generalization of 
the natural sum in which the supremum is taken at the  limit stage
 has been considered in
Wang \cite{W}  and 
 V\"a\"an\"anen and  Wang \cite{VW}
with applications
to infinitary logics. This  infinite natural sum has been studied in detail in
\cite{w}, where an order-theoretical characterization has been provided.
In the present note we introduce and study the analogously
 defined infinitary product.
The computation of 
the infinite natural product 
can be  reduced  to the computation
of some---possibly infinite---natural sum;
in particular, we can directly  transfer results
from sums to products, rather than repeating
essentially the same arguments.
 Curiously, the method  applies 
to the usual infinitary operations, too. See 
Theorem \ref{ordinary}
and Corollary \ref{corsier}; we are not aware of previous
uses of 
this technique. 

Order-theoretical characterizations of the finitary natural operations
have been provided in  Carruth \cite{Car}.
Some characterizations seem to have been independently rediscovered many 
times, e.~g., Toulmin \cite{T} and
  de Jongh and  Parikh \cite{dJP}.
As we showed in \cite{w},
 Carruth  characterization of the finite natural sum cannot be generalized as it stands
to the infinitary natural sum; see, in particular,
the comments at the beginning of \cite[Section 4]{w}. 
A similar situation occurs for the infinitary natural product; 
see Subsection \ref{4s} below. 
In the case of the infinite natural sum
 the difficulty can be circumvented by imposing 
a finiteness condition to a Carruth-like description: see 
\cite[Theorem 4.7]{w}.
As we shall show in Section \ref{otc}, a  similar result holds 
for the infinite natural product, though
the situation gets
 technically more involved.

To explain our construction in some detail,
it is well known that the usual finite ordinal product
is order-theoretically characterized by taking the reverse lexicographical order
on the product of the factors.
Less known,  the same  holds in the
infinite case, too, provided 
that in the product one takes into account only elements with 
finite support. See Hausdorff \cite[\S\ 16]{Ha} and  Matsuzaka \cite{M}.
We show that the infinite natural product can be characterized by merging the 
representation by  Carruth  and the  just mentioned one; roughly,
by using Carruth order for a sufficiently large but finite product  of factors
and  then working as in an ordinary product
 when we approach  infinity.
Quite surprisingly, a local version of the result holds.
Namely, if $\leq'$ is a linear order on the finite-support-product
of a sequence $( \alpha_i) _{i < \omega} $ of ordinals 
and $\leq'$ 
is such that,
for every element $c$, the set of the $\leq'$ predecessors of $c$
is built in a way similar to above, then 
$\leq'$ is a well-order of order type
less than or equal to the infinite natural product of the 
$\alpha_i$'s.  Moreover, the infinite natural product 
is actually the  maximum of the ordinals obtained this way, that is,
the order-type of the product can  always be realized in the above way.
See Section \ref{otc}
and in particular Theorem \ref{otcth} for full details.

\subsection{Preliminaries} \labbel{prel} We assume familiarity with the 
basic theory of ordinal numbers.
Unexplained notions and notations are standard and
can be found, e.~g., in the mentioned 
books \cite{Bac,Ha, Sier}.
Notice  that here
sums, products and exponentiations  are always considered in
 the ordinal sense.
The usual ordinal product of two ordinals $\alpha$ and $\beta$ is denoted 
by $ \alpha \beta $
or sometimes $\alpha \cdot \beta $ for clarity.
The classical infinitary ordinal product is denoted by
  $\prod _{i < \omega }  \alpha _i $. 
Recall that every nonzero ordinal $ \alpha $ can be 
expressed in \emph{Cantor normal form} in
a unique way  as follows 
\begin{equation*} \labbel{cantor}   
 \alpha  =
 \omega ^ {\xi_k} r_k + \omega ^ {\xi _{k-1}}r _{k-1}    +
\dots
+ \omega ^ {\xi_1} r_1 + \omega ^ {\xi_0}r_0  
 \end{equation*}
  for some  integers 
$k \geq 0$, $r_k, \dots, r_0 >0$ and ordinals
$ \xi _k > \xi _{k-1} > \dots > \xi_1 > \xi_0  $.
If $\beta$ is another ordinal expressed in Cantor normal form,
the \emph{natural sum} $ \alpha \+ \beta $ is obtained by summing
the two expressions as if they were polynomials in $ \omega$.
The \emph{natural product} $\alpha \x \beta $    
is computed by using the rule
$ \omega^ \xi \x \omega ^ \eta = \omega ^{\xi \+ \eta} $ and
then expanding again by  ``linearity''.  
Both $\+$ and $\x$ are commutative, associative and cancellative
(except for multiplication by $0$, of course). 
If $( \alpha_i) _{i < \omega} $ is a sequence of ordinals,
$\nsum _{i < \omega} \alpha_{i} = \sup _{n< \omega } ( \alpha _0 \+ \alpha _1 
\+ \dots \+ \alpha _{n-1} )$.
See \cite{W,VW,w} for further details about $\nsum$.

\section{An infinite natural product} \labbel{ainp} 

\begin{definition} \labbel{dfnp}
Suppose that  $( \alpha_i) _{i < \omega} $  is a sequence of ordinals.

For $i < \omega $, let $P_i$
denote the partial natural product
$ \alpha _0 \x \alpha _1 \x \dots \x \alpha _{i-1} $,
with the convention 
$P_0=1$.  

 Define
$ \xx  _{i < \omega} \alpha_i  = \lim _{i < \omega } P_i $.

This means that  
$ \xx  _{i < \omega} \alpha_i = 0$ 
if at least one $\alpha_i$ is $0$ and  
$ \xx  _{i < \omega} \alpha_i  = \sup _{i < \omega } P_i $
if each $\alpha_i$ is different from $0$
(cf. Clause (5) in Proposition \ref{factswp} below). 
\end{definition}   

From now on, when not otherwise specified, 
$( \alpha_i) _{i < \omega} $ is a fixed sequence of ordinals 
and the partial natural
products $P_i$ are always computed as above and   with respect to
the sequence 
$( \alpha_i) _{i < \omega} $.

We now state some simple facts 
about the infinitary operation $\xx$.
Only (2), (5) and (6) will be used in what follows. 

\begin{proposition} \labbel{factswp}
Let $\alpha_i$, $\beta_i$ be two sequences of 
ordinals and $ n,m < \omega$.  
  \begin{enumerate}    
\item
$\prod_{i < \omega  } \alpha _i
\leq
\xx _{i < \omega   } \alpha _i$
\item
If $\beta_i \leq \alpha _i$, for every $ i < \omega$,
then
$\xx _{i < \omega   } \beta  _i \leq \xx _{i < \omega   } \alpha _i$    
\item
If  $\pi$ is a permutation of $ \omega$, then
$\xx _{i < \omega } \alpha _i = \xx _{i < \omega } \alpha _{\pi(i)}$
\item
More generally, suppose that $(F_h) _{h < \omega } $ is a partition
of $ \omega$ into finite subsets,
say, $F_h = \{ j_1, \dots, j _{r(h)} \} $, for every 
$ h \in \omega$.  Then
\begin{equation*}
\xx _{i < \omega   } \alpha _i =
\xx _{h < \omega   } \,\xx _{j \in F_h}  \alpha _j=
\xx _{h < \omega   } ( \alpha _{j_1} \x \alpha _{j_2} \x \dots \x \alpha _{j _{r(h)}}  )
  \end{equation*} 
 \end{enumerate} 
 Suppose in addition that $\alpha_i \neq 0$, for every $i < \omega$.
\begin{enumerate}  
\item[(5)]
If $i < j$, then
$P_i \leq P_j$; equality holds if and only if 
$\alpha_i= \dots= \alpha _{j-1} =1$.   
\item[(6)]
For every $n< \omega$, we have 
$P_n  \leq \xx _{i < \omega } \alpha _i$; equality holds
if and only if $\alpha_i=1$, for every $i \geq n$.  
   \end{enumerate}
 \end{proposition} 

\begin{proof}
Easy, using the properties 
of the finitary $\x$. We just comment on (3).
This is trivial if at least one $\alpha_i$ is $0$. 
Otherwise, every partial product on the left is bounded by some
sufficiently long partial product on the right (by monotonicity,
associativity and commutativity of $\x$, 
and since all factors are nonzero by assumption); the converse holds as well,
hence the infinitary products are equal. 
 
Let us mention that (3) and (4) can be also  proved by using
Theorem \ref{descr} below and the corresponding properties of $\nsum$
given in \cite[Proposition 2.4(5)(6)]{w}.   
\end{proof}

\begin{lemma} \labbel{lem}
If $( \beta _i) _{i < \omega} $  is a sequence of ordinals
and $\beta= \nsum _{i < \omega } \beta _i$, then 
\begin{equation*}\labbel{elem}
\xx _{i < \omega } \omega ^{ \beta _i} = \omega ^{\beta }   
  \end{equation*}    
 \end{lemma}

 \begin{proof} 
Since $\omega ^{ \beta _i}$
is always different from $0$,
we  have  
 $ \xx  _{i < \omega} \alpha_i  = \sup _{i < \omega } P_i $.
Here, of course,  $P_i$ is computed with respect to the sequence
given by $\alpha_i = \omega ^{ \beta _i}$, for $ i< \omega$. 

By a property of the natural product
(or the definition, if you like), we have
$P_i= 
\omega ^{ \beta _0} \x  \omega ^{ \beta _1} \x \dots \x  \omega ^{ \beta _{i-1}} =
 \omega ^{ \beta _0 \+ \beta _1 \+ \dots \+ \beta _{i-1} }$,
for every $ i < \omega$.  
Letting $B_i = \beta _0 \+ \beta _1 \+ \dots \+ \beta _{i-1} $,
we have by definition 
 $\beta= \nsum _{i < \omega } \beta _i = \sup  _{i < \omega } B _i $.
But then 
 $ \xx  _{i < \omega} \alpha_i  = \sup _{i < \omega } P_i 
= \sup _{i < \omega }  \omega ^{B_i} =
\omega ^{ \beta }$
by continuity of the exponentiation.
\end{proof}

Let $\alpha$ be a nonzero ordinal expressed as
$ \omega ^ {\xi_k} r_k +
\dots
 + \omega ^ {\xi_0}r_0  $
 in Cantor normal form. The ordinal $d( \alpha ) =\xi_ k$ 
will be called  the \emph{degree}
or the \emph{largest exponent} of $\alpha$.
The ordinal $ m ( \alpha ) = \omega ^ {\xi_k} r_k$
will be called the \emph{leading monomial}
of $\alpha$.  By convention, we set $m(0) = 0$.
The following lemma is trivial, but it  will be useful
in many situations.

\begin{lemma} \labbel{lema}
For every ordinal $\alpha$,
\begin{equation*}\labbel{brup}   
m(\alpha) \leq \alpha 
\leq m(\alpha) + m(\alpha) =  m( \alpha ) \cdot 2
\leq m(\alpha) \x 2
 \end{equation*}  
 \end{lemma} 

\begin{lemma} \labbel{lem2}
If  $( \alpha  _i) _{i < \omega} $  is a sequence of ordinals
which is not eventually $1$, then 
\begin{equation*}\labbel{prik} 
 \xx  _{i < \omega} \alpha_i =   \xx  _{i < \omega} m(\alpha_i)  
 \end{equation*}    
 \end{lemma}

 \begin{proof}
This is trivial if at least 
one $\alpha_i$ is $0$. 

Hence suppose that 
$\alpha_i \neq 0$, for all $i< \omega$, 
 and that $( \alpha_i) _{i < \omega} $  is not eventually $1$.
The inequality 
$ \xx  _{i < \omega} \alpha_i \geq
  \xx  _{i < \omega} m(\alpha_i) $ 
is trivial by monotonicity (Proposition \ref{factswp}(2)), since  
$\alpha \geq m(\alpha) $,
for every ordinal $\alpha$. 

To prove the converse,
we show that, for every $h < \omega$, 
there is $k < \omega$ 
such that 
$ \xx  _{i < h} \alpha_i \leq
  \xx  _{i < k} m(\alpha_i) $.
This is enough since if all the $\alpha_i$'s are nonzero,
then the succession of the partial products is nondecreasing
(Proposition \ref{factswp}(5)).
It is a trivial property of the finitary natural 
product that 
$m( \xx  _{i < h} \alpha_i) =
  \xx  _{i < h} m(\alpha_i) $
(by strict monotonicity of $\+$).
Since $( \alpha_i) _{i < \omega} $ 
is not eventually $1$,
there is  $k > h$ such that $\alpha_{k-1} \geq 2$.  
Then 
$  \xx  _{i < k} m(\alpha_i) \geq 
\left( \xx  _{i < {k-1}} m(\alpha_i) \right) \x 2 \geq
\left( \xx  _{i < h} m(\alpha_i) \right) \x 2 =
m \left( \xx  _{i < h} \alpha_i \right) \x 2  \geq
 \xx  _{i < h} \alpha_i
$, by Lemma \ref{lema}. 
 \end{proof}

\begin{theorem} \labbel{descr}
Let  $( \alpha  _i) _{i < \omega} $  be a sequence of ordinals
and let $\beta= \nsum _{i < \omega } d( \alpha _i) $. 
The infinite natural product $ \xx  _{i < \omega} \alpha_i  $
can be computed according to the following rules. 
\begin{align}\labbel{a1}  
 \xx  _{i < \omega} \alpha_i & = 0 \quad 
\text{ if (and only if) at least one $\alpha_i$ is equal to } 0; \\ 
\labbel{a2} 
\xx  _{i < \omega} \alpha_i & = \alpha _0 \x \dots \x \alpha_{n-1} 
\quad \quad
\text{ if $ \alpha _i = 1$, for every $i \geq n$};   \\    
\labbel{a3} 
\begin{split}    
 \xx  _{i < \omega} \alpha_i  &=
 \omega ^{d( \alpha _0) \+ \dots \+ d(\alpha _{n-1}) + 1} =
  \omega^ {\beta+ 1} 
\quad \text{ if $\alpha_i \neq 0$, for all $i< \omega $,} \\
\alpha_i & < \omega  \text{, for all $i \geq  n$,
and the sequence is not
 eventually $1$;}  
\end{split}   
\\
\labbel{a4}   
 \xx  _{i < \omega} \alpha_i &=  
 \xx  _{i < \omega}  \omega ^{d(\alpha_i)}=
 \omega^ {\beta}  
\quad \text{ if none of the above cases applies,} 
\end{align} 
\quad \quad \quad \quad that is, no element of the sequence is $0$ 
and the members
of 

\quad \quad \quad  the  sequence are not
eventually $< \omega $.
 \end{theorem}

Before proving Theorem \ref{descr}, we notice that it
gives an effective way to compute
$\xx  _{i < \omega} \alpha_i$,
for every sequence $( \alpha_i) _{i < \omega} $ 
of ordinals. 
Apply \eqref{a1}  if at least one $\alpha_i$ is equal to  $0$;
if this is not the case, 
apply \eqref{a2}  if the $\alpha_i$ are  eventually $1$;
if not, then exactly one of \eqref{a3} or \eqref{a4} 
 occurs.
Notice that  conditions \eqref{a1} and \eqref{a2} in Theorem \ref{descr} 
might  overlap, and $n$ in 
\eqref{a2} and \eqref{a3} is not uniquely defined,
 but the conditions give the same outcome
in any overlapping case.
Notice also that the expression
$d( \alpha _0) \+ \dots \+ d(\alpha _{n-1}) + 1$ in 
\eqref{a3}  causes no ambiguity, since
$(d( \alpha _0) \+ \dots \+ d(\alpha _{n-1})) + 1=
d( \alpha _0) \+ \dots \+ (d(\alpha _{n-1}) + 1)$. 

 \begin{proof}
The result follows trivially from the definitions  if
some $\alpha_i$ is equal to $0$
or when the sequence is eventually $1$. 

If we are in the case given by \eqref{a3}, then,
for every $ \ell < \omega$, 
there is  $h > n$ such that there are
at least $ \ell$-many $\alpha_i \geq 2$,
where the index $i$ varies between $n$ and $h$.    
Thus
 $ \xx  _{i < h+1} \alpha_i  \geq 
 \omega ^{d( \alpha _0)} \x \dots \x \omega  ^{d(\alpha _{n-1})} \x 2^ \ell $
since $\alpha \geq \omega ^{d( \alpha )}$,
for every nonzero ordinal $\alpha$.
Since $\ell$ is arbitrary,
we get
 $ \xx  _{i < \omega} \alpha_i  
= \sup _{h< \omega } \xx  _{i < h+1} \alpha_i 
\geq 
\sup _{ \ell < \omega }
 ( \omega ^{d( \alpha _0)} \x \dots \x \omega  ^{d(\alpha_{n-1})} \x 2^ \ell )
=
\sup _{ \ell < \omega }
 \omega ^{d( \alpha _0) \+ \dots \+ d(\alpha _{n-1})} 2^ \ell =
 \omega ^{d( \alpha _0) \+ \dots \+ d(\alpha _{n-1}) + 1}$,
since $ \omega^ \varepsilon \x 2 =
 \omega^ \varepsilon  2$ and
$\sup _{p < \omega } \omega ^ \varepsilon p = 
 \omega ^{ \varepsilon +1}   $,
for every ordinal $\varepsilon$.

In the other direction, 
we have 
$ \xx  _{i < \omega} \alpha_i =   \xx  _{i < \omega} m(\alpha_i)  $
from Lemma \ref{lem2},
hence it is enough to prove 
$\xx  _{i < \omega} m(\alpha_i)  \leq \allowbreak
 \omega ^{d( \alpha _0) \+ \dots \+ d(\alpha _{n-1}) + 1}$.  
 If $h < \omega$, and, for every $ i < \omega$,
 letting  $s_i$ be the only natural number such that 
$m(\alpha_i)  =   \omega ^{d( \alpha _i)} s_i$, 
then by associativity and commutativity of $\x$, we get  
$\xx  _{i < h+1} m(\alpha_i)  =
 \omega ^{d( \alpha _0)} \x
s_0 \x  \dots \x \omega  ^{d(\alpha _h)} \x
 s_h 
=
 \omega ^{d( \alpha _0)} \x
 \dots \x \omega  ^{d(\alpha _h)} \x
s_0 \x  \dots \x  s_h 
\leq 
 \omega ^{d( \alpha _0)} \x \dots \x \omega  ^{d(\alpha _{n-1})} \x \omega =
 \omega ^{d( \alpha _0) \+ \dots \+ d(\alpha _{n-1}) + 1}
$,
since $d( \alpha _i) =0$,
for $i \geq n$ and, by construction,
$s_i < \omega $, for every $ i < \omega$.  
Hence
$\xx  _{i < \omega} m(\alpha_i)  =
\sup _{h < \omega } \xx  _{i < h} m(\alpha_i)  \leq
 \omega ^{d( \alpha _0) \+ \dots \+ d(\alpha _{n-1}) + 1}
$.

The last identity in  \eqref{a3} follows from the already mentioned fact that
 $d( \alpha _i) =0$,
for $i \geq n$, hence 
$\beta= \nsum _{i < \omega } d( \alpha _i) =  
\nsum _{i < n} d( \alpha _i) $.

The case given by 
  \eqref{a4} is similar and somewhat easier.
The inequality $ \xx  _{i < \omega} \alpha_i \geq  
 \xx  _{i < \omega}  \omega ^{d(\alpha_i)}$ 
is trivial by monotonicity.

For the converse, we use again the identity
$ \xx  _{i < \omega} \alpha_i =   \xx  _{i < \omega} m(\alpha_i)  $
from Lemma \ref{lem2}.
Arguing as in case   \eqref{a3},
we have that, for every $h< \omega$,
$\xx  _{i \leq  h} m(\alpha_i)  
\leq 
 \omega ^{d( \alpha _0)} \x \dots \x \omega  ^{d(\alpha _h)} \x \omega
$,
but there is some  $k > h$ such that 
 $\alpha_k \geq  \omega $, 
since the members of the sequence are not
eventually $< \omega $.
Hence
 $\xx  _{i \leq h} m(\alpha_i)  
\leq 
 \omega ^{d( \alpha _0)} \x \dots \x \omega  ^{d(\alpha _h)} \x \omega
\leq
 \omega ^{d( \alpha _0)} \x \dots \x 
 \omega  ^{d(\alpha _h)} \x
\dots \x\omega  ^{d(\alpha _k)}
\leq 
\xx  _{i < \omega}  \omega ^{d(\alpha_i)} 
$.
In conclusion,
$ \xx  _{i < \omega} \alpha_i =   \xx  _{i < \omega} m(\alpha_i) =
\sup _{h < \omega } \xx  _{i \leq h} m(\alpha_i)   \leq
\xx  _{i < \omega}  \omega ^{d(\alpha_i)}  $.

The last identity is from Lemma \ref{lem}.
 \end{proof}

Corollary 5.1 in \cite{w}
can be used to provide a more precise evaluation of $\beta$ in case 
\eqref{a4}  in Theorem \ref{descr}.

\begin{corollary} \labbel{cordescr}
Suppose that $( \alpha_i) _{i < \omega} $ is a sequence of ordinals
such that no element of the sequence is $0$ 
and the members
of 
 the  sequence are not
eventually $< \omega $
(thus the sequence  $( d(\alpha_i)) _{i < \omega} $
is not eventually $0$).
Let $\xi$ be the smallest ordinal such that 
$\{ i < \omega \mid d( \alpha _i) \geq \omega ^\xi \}$
is finite, and enumerate 
those $\alpha_i$'s such that  
 $d( \alpha _i) \geq \omega ^\xi $ as 
$\alpha _{i_0}, \dots, \alpha _{i_k} $
(the sequence might be empty).
Then 
$ \xx  _{i < \omega} \alpha_i =
 \omega^ {\beta} $,    
where
$ \beta = (d(\alpha _{i_0}) \+ \dots \+ d(\alpha _{i_k})) + \omega ^\xi $.
\end{corollary}

We need the results analogous to
Theorem \ref{descr} for the classical ordinal product. 

\begin{theorem} \labbel{ordinary} 
  \begin{enumerate} 
   \item   
If $( \beta _i) _{i < \omega} $  is a sequence of ordinals
and $\beta= \sum _{i < \omega } \beta _i$, then 
\begin{equation*}\labbel{elemordi}
\prod _{i < \omega } \omega ^{ \beta _i} = \omega ^{\beta }   
  \end{equation*} 
\item
If  $( \alpha  _i) _{i < \omega} $  is a sequence of ordinals
which is not eventually $1$, then 
\begin{equation*}\labbel{prikordi} 
 \prod  _{i < \omega} \alpha_i =   \prod  _{i < \omega} m(\alpha_i)  
 \end{equation*}    
\item
Suppose that $( \alpha  _i) _{i < \omega} $  is a sequence of nonzero ordinals
which is not eventually $1$ 
and let $\beta= \sum _{i < \omega } d( \alpha _i) $. Then
\begin{align*} 
 \prod  _{i < \omega} \alpha_i  &=
 \omega ^{d( \alpha _0) + \dots + d(\alpha _{n-1}) + 1} =
  \omega^ {\beta+ 1} 
\quad \text{ if 
 $\alpha_i < \omega $, for all $i \geq  n$}  \\ 
 \prod  _{i < \omega} \alpha_i &=  
 \prod _{i < \omega}  \omega ^{d(\alpha_i)}=
 \omega^ {\beta}  
\quad \text{ if the sequence is not eventually $< \omega $}
\end{align*}    
\end{enumerate} 
\end{theorem}

\begin{proof}
(1) Like the proof of Lemma \ref{lem},
using the identity  
$ \omega ^{ \beta _0} \omega ^{ \beta _1} \dots  \allowbreak  \omega ^{ \beta _{i-1}} 
\allowbreak =
 \omega ^{ \beta _0 + \beta _1 + \dots + \beta _{i-1} }$.

(2) Like the proof of Lemma \ref{lem2}.
In fact, we do have 
$m( \prod _{i < h} \alpha_i) = 
  \prod _{i < h} m(\alpha_i) $
and this is enough for the proof.

(3) is proved as Theorem \ref{descr}.
 \end{proof}

Theorems \ref{descr} and \ref{ordinary} 
can be used to
``lift'' 
some results from infinitary sums to infinitary 
products.
To show how the method works,
we first present a simple example,
though only feebly connected with the rest of this note. 

Sierpi{\'n}ski \cite{Sie} 
showed that a sum $\sum _{i < \omega} \alpha_{i}$ 
of ordinals can assume only finitely many values,
by permuting the $\alpha_i$'s.
A proof can be found also in \cite{w}.
Then Sierpi{\'n}ski  in \cite{Sie2} showed  the analogous result for an 
infinite product. 
Using Theorem \ref{ordinary}  we show that the result
 about products is immediate
from the result about sums. 

\begin{corollary} \labbel{corsier} 
If $( \alpha  _i) _{i < \omega} $ is a sequence  of ordinals,
one obtains only a finite number of ordinals by considering
all products of the form  $\prod _{i \in I} \gamma _i $,
where $( \gamma _i) _{i < \omega} $
is a \emph{permutation} of
 $( \alpha_i) _{i < \omega} $,
that is, there exists a bijection
$ \pi : \omega  \to \omega $   
such that $\gamma_i =  \alpha_ {\pi(i)} $
for every $i < \omega$. 
\end{corollary}

\begin{proof}
This is trivial if some $\alpha_i$
is $0$, or if the $\alpha_i$'s are eventually $1$,
so let us assume that none of the above cases occurs.

If the sequence is eventually $< \omega $,
then the first equation in Theorem \ref{ordinary}(3)
shows that we obtain only a finite number of products by taking rearrangements
of the factors, since the resulting products are
given by $ \omega ^{ \delta +1} $, where 
$\delta$ is a sum of
$d( \alpha _0) , \dots, d( \alpha _{n-1})$, taken in some order,
but there is only a finite number of rearrangements of this finite set,
hence there are only a finite number of possibilities for $\delta$
(notice that if $\alpha_i< \omega $, then 
$d(\alpha_i)=0$, hence the degrees of finite ordinals
do not contribute to the sum).

In the remaining case, the products obtained by rearrangements 
have the form
$ \omega^ \delta $,
with $\delta = \sum _{i < \omega } d( \gamma _i) $,  
by the second equation
in Theorem \ref{ordinary}(3).    
The quoted result from \cite{Sie} shows that we have only a
 finite number of possibilities for $\delta$,
hence there are only a finite number of possibilities for the 
values of the
rearranged products.
\end{proof}

We now use Theorems \ref{descr} and \ref{ordinary} 
in a slightly more involved situation 
in order to transfer some results from \cite{w} about infinite natural sums
to results about infinite natural products.

\begin{corollary} \labbel{segue}
For every sequence $( \alpha  _i) _{i < \omega} $  of ordinals
 there is $ m < \omega$ such that, for every $n \geq m$,
\begin{align} \labbel{1} 
 \xx_{n\leq i < \omega } \alpha _i 
& =
 \prod_{n\leq i < \omega } \alpha _i  \quad \text{ and } 
\\ 
\begin{split}    \labbel{2}
 \xx _{ i < \omega } \alpha _i
& =
( \alpha _0 \x \dots \x \alpha _{n-1}) \cdot 
  \xx _{n\leq i < \omega } \alpha _i \\
 & =
( \alpha _0 \x \dots \x \alpha _{n-1}) \cdot 
  \prod_{n\leq i < \omega } \alpha _i 
 \end{split}
 \end{align}
 and if, moreover, every $\alpha_i$ is nonzero
and the sequence is not eventually $1$, then
\begin{align} \labbel{eqp}
\begin{split} 
\xx _{ i < \omega } \alpha _i
& =
  \omega ^{ \beta  _0 \+ \dots \+ \beta  _{n-1}} \cdot
  \xx _{n\leq i < \omega } \alpha _i \\
 & =
  \omega ^{ \beta  _0 \+ \dots \+ \beta  _{n-1}} \cdot
  \prod_{n\leq i < \omega } \alpha _i 
\end{split} 
\end{align}
where $\beta_i = d( \alpha _i)$,
for every $i < \omega$. 
\end{corollary}

\begin{proof}
The result is trivial if the sequence is eventually $1$;
moreover, (6) is  trivial
if some $\alpha_i$ is $0$. Furthermore,
(5) is trivial if, for every $i< \omega $, there is $j>i$ such that 
$\alpha_j=0$. Otherwise, by taking $m$ large enough,
we have $ \alpha _i >0$,  for $i>m$.
Henceforth it is enough to prove the result in the case when 
the sequence is not eventually $1$ and all the $\alpha_i$'s
are nonzero.       

We shall first prove \eqref{1} and  \eqref{eqp}
and then derive \eqref{2}. 
If the sequence is eventually $< \omega$,
then \eqref{1} is trivial, since in this case, for large enough $n$,
both sides are equal to $ \omega$. Then
\eqref{eqp} is immediate from equation \eqref{a3} in 
Theorem \ref{descr}, since    
$ \omega ^{ \beta  _0 \+ \dots \+ \beta  _{n-1}} \omega =  
\omega ^{ \beta  _0 \+ \dots \+ \beta  _{n-1}+1}$.  

Suppose now that the sequence  $( \alpha_i) _{i < \omega} $ is not eventually $< \omega$, hence the sequence $( \beta _i) _{i < \omega} $ 
is not eventually $0$.
By \cite[Theorem 3.1]{w}, 
there is $ m < \omega$ such that, for every $n \geq m$,
we have 
$\nsum _{n\leq i < \omega } \beta _i 
=
 \sum _{n\leq i < \omega } \beta _i $.
Fixing some 
$n \geq m$ and 
letting $ \beta ' = \nsum _{n\leq i < \omega } \beta _i $,
we get
$\xx _{n\leq i < \omega } \alpha  _i  = \omega ^{ \beta '}
=  \prod _{n\leq i < \omega } \alpha  _i$ 
from, respectively, 
equation \eqref{a4} in Theorem \ref{descr}  
and the last equation in 
Theorem \ref{ordinary}(3). 
This proves 
 \eqref{1}.

Letting 
$ \beta  = \nsum _{ i < \omega } \beta _i $,
we notice that in  \cite[Theorem 3.1]{w}
it has also been proved that
$ \beta =
  ( \beta _0 \+ \dots \+ \beta _{n-1} )
+
\beta '$.
Using the above identity and applying 
equation \eqref{a4} in Theorem \ref{descr}  twice, we get
$\xx _{i < \omega } \alpha  _i  = \omega ^{ \beta } =
  \omega ^{ \beta  _0 \+ \dots \+ \beta  _{n-1}} \omega ^{ \beta '}=
  \omega ^{ \beta  _0 \+ \dots \+ \beta  _{n-1}}   \xx _{n\leq i < \omega } \alpha _i $,
that is, \eqref{eqp} (the second identity in  \eqref{eqp} could
be proved in the same way, but now it follows immediately
from \eqref{1}).  

Equation \eqref{2} remains to be proved.
We shall prove that in the nontrivial cases
the expressions given by 
\eqref{2} and \eqref{eqp} are equal.
One direction is trivial, since
$  \omega ^{ \beta  _0 \+ \dots \+ \beta  _{n-1}}=
  \omega ^{ \beta  _0} \x \dots \x \omega ^{ \beta  _{n-1}}
\leq 
 \alpha _0 \x \dots \x \alpha _{n-1}$.
For the other direction, let us observe that, in 
the nontrivial cases,
again by Theorem \ref{descr}, 
$  \xx _{n\leq i < \omega } \alpha _i$ has the form 
$ \omega^ \beta $, for some $\beta \geq 1$.
Then
$( \alpha _0 \x \dots \x \alpha _{n-1}) \cdot  \xx _{n\leq i < \omega } \alpha _i
=
( \alpha _0 \x \dots \x \alpha _{n-1}) \cdot  \omega ^ \beta 
\leq
m( \alpha _0 \x \dots \x \alpha _{n-1}) \cdot 2 \cdot  \omega ^ \beta 
=
(m( \alpha _0) \x \dots \x m(\alpha _{n-1})) \cdot  \omega ^ \beta 
=
(\omega ^{d( \alpha _0)}s_0 \x \dots \x \omega ^{d( \alpha _{n-1}) }s _{n-1})
 \cdot \omega ^ \beta 
=
(\omega ^{ \beta  _0} \x \dots \x  \omega ^{ \beta  _{n-1}}   \x s_o \dots s _{n-1})
  \cdot  \omega ^ \beta 
=
\omega ^{ \beta _0 \+ \dots \+ \beta _{n-1} } \cdot  s_o \dots s _{n-1}
  \cdot  \omega ^ \beta 
=
\omega ^{\beta _0 \+ \dots \+ \beta _{n-1}} 
  \cdot  \omega ^ \beta 
=
 \omega ^{\beta _0 \+ \dots \+ \beta _{n-1}}  
\cdot  \xx _{n\leq i < \omega } \alpha _i$,
where we used Lemma \ref{lema} and the facts that 
$k \omega ^ \beta = \omega ^ \beta$,
whenever  $ k < \omega$ and $\beta \geq 1$,
and that
$ \omega^ \xi  \x k = \omega^ \xi  \cdot  k$,
for all  ordinals $ k < \omega$ and $\xi$.
\end{proof}

\section{An order-theoretical characterization} \labbel{otc} 

\subsection{} \labbel{1s} We refer to, e.~g., Harzheim \cite{H} for a general 
reference about ordered sets.
As usual, when no risk of ambiguity is present, we shall 
denote a (partially) ordered set $(P, \leq )$ simply as $P$.
However, in many situations, we shall have several different orderings
on the same set; in that case we shall
explicitly indicate the order. It is sometimes convenient to define $\leq$ 
in terms of the associated $<$ relation and conversely.
As a standard convention,
$a \leq b$ is equivalent to ``either $a=b$ or $a < b$''
(strict disjunction).
We shall be quite informal about the distinction
and we shall use either $\leq$ or $<$ 
case by case according to convenience,
even when we are dealing with (essentially) the same order.

In order to avoid notational ambiguity, let us denote the cartesian product 
of a family $( A_i) _{i \in  I} $ of sets by 
$\bigtimes _{i \in I } A_i$.
If each $A_i$ is an ordered set, with the order denoted by $\leq_i$,
then  a partial order $\leq_{\times}$  
can be defined on $\bigtimes _{i \in I } A_i$ 
componentwise.
Namely, we put 
$a \leq_{\times} b$ if and only if 
$a_i \leq_i b_i$, for every $i \in I$.  
Apparently, for our purposes, the above definition has little use 
when dealing with ordinals
(more precisely, order-types of well-ordered sets),
since ordinals are linearly ordered, but
generally the above construction furnishes only a partially ordered set.
However, see below for uses of 
the ordered set $( \bigtimes _{i \in I } A_i, \leq_{{\times}})$.

\subsection{} \labbel{2s} An order theoretical characterization
of the ordinal product can be given using 
lexicographic products.
If $(I , \leq _I)$ is reverse-well-ordered
and  each $A_i$ is an ordered set,
then the \emph{anti-lexicographic order}
  $\LL {i \in I } A_i = 
(\LL {i \in I } A_i , \leq_* )$ on the set $\bigtimes _{i \in I } A_i$ 
is obtained by putting 
  \begin{enumerate}    \item []
$a \leq_* b$ if and only if  either $a = b $, or
$a_i <_i  b_i$, where $i$ is the largest element of $I$ such that 
$a_i \neq b_i$.
   \end{enumerate} 
In other words, $\leq_*$ orders 
  $\LL {i \in I } A_i $ by the \emph{last} difference. 
The definition makes sense, since $I$ is reverse-well-ordered.
It turns out that if each $A_i$ is linearly ordered,
then    $\LL {i \in I } A_i$ is linearly ordered and
if in addition $I$ is finite, then   $\LL {i \in I } A_i$ is well-ordered.
Moreover, for two ordinals $\alpha_0$ and
$\alpha_1$, it happens that  $\alpha_0 \alpha _1$ 
is exactly the order-type of
 $\LL {i < 2} \, \alpha _i$.
This can be obviously generalized to \emph{finite} products.

A similar characterization
can be given for  infinite products,
but some details should be made precise.
Suppose that 
 each $A_i$ is an ordered set with order $\leq_i$ and  
with a specified element $0_i \in A_i$. 
If $ a \in \bigtimes _{i \in I } A_i$,
the \emph{support} $supp(a)$ of $ a$ 
is the set $\{ i \in I \mid a_i \neq 0_i \}$.  
Let $\bigtimes _{i \in I }^0 A_i $
be the subset of 
$ \bigtimes _{i \in I } A_i$
consisting of those elements
with finite support. 
Of course, $\bigtimes _{i \in I }^0 A_i $
inherits a partial  order $\leq _{\times}^0$ as a suborder of 
$ (\bigtimes _{i \in I } A_i , \leq _{\times})$.
If $I$ is linearly ordered,
we can consider  another  order   $ \LLL {i \in I } A_i =
(\LLL {i \in I } A_i, \leq _L)$  on 
the set $\bigtimes _{i \in I }^0 A_i $ defined as follows.
  \begin{enumerate}    \item [(*)]
$a \leq_L b$ if and only if  either $a = b $, or
$a_i <_i b_i$, where $i$ is the largest element of 
$supp(a) \cup supp(b)$ 
 such that 
$a_i \neq b_i$.
   \end{enumerate} 

The definition makes sense, since $supp(a) \cup supp(b)$ 
is finite and $I$ is linearly ordered. 
Of course, when $I$ is finite, 
$\bigtimes _{i \in I }^0 A_i $ 
and $\bigtimes _{i \in I } A_i $ are the same set and 
$\LLL {i \in I }$ and $\LL {i \in I }$ are the same order. 
It is known that,
for every sequence 
$( \alpha_i) _{i < \delta } $
of ordinals, 
  $\LLL {i <\delta  } \, \alpha _i$
is well-ordered and has order-type
$\prod _{i <  \delta }  \alpha _i $.
Here the specified element $0_i$ is always 
chosen to be the ordinal $0$.
See 
\cite[III, \S\ 10 1.2 and \S\ 11 Satz 7]{Bac},
Hausdorff \cite[\S\ 16]{Ha} and Matsuzaka \cite{M} for details. 
Of course,  we could have seen just by cardinality considerations 
that  $\bigtimes _{i < \delta  } A_i $
does  not work in order 
to obtain  an order theoretical characterization of
$\prod _{i  <  \delta }  \alpha _i $
in the  case when $\delta$ is infinite.

\subsection{} \labbel{3s} 
Dealing now with  natural products,
a characterization in the finite case has been 
found by Carruth \cite{Car}.
He proved  that 
$\alpha_0 \otimes \alpha _1 $ 
is the largest ordinal
which is the order-type of some  linear extension 
of the componentwise order 
on $\alpha_0 \times \alpha _1 $. 
Here   $\alpha_0 \times \alpha _1 $
is ordered as   \mbox{$(\bigtimes _{i < 2} 
 \alpha _i , \leq _ \times )$}
in the above notation.
Carruth result includes the proof 
that such an ordinal exists,
that is, 
that each  such linear extension is  a well-order and that
the set of order-types
of such linear extensions has a maximum,
not just a supremum.
From the modern point of view, 
this can be seen as a special case of 
theorems by Wolk \cite{Wo} and 
de Jongh, Parikh \cite{dJP}, since the product of
two well quasi-orders
 (in particular, two well-orders)
 is still a well quasi order. 
See, e.~g.,  \cite{Mi}. 
Carruth result 
obviously extends to the case of any finite number of factors.

\subsection{} \labbel{4s} 
Some difficulties are encountered
when trying to unify the  results
recalled in \ref{2s} and  \ref{3s}.
Let us limit ourselves to the simplest infinite case 
of an $ \omega$-indexed sequence, 
which is the main theme of the present note.
It would be natural to 
consider
the supremum of the order-types
of well-ordered linear extensions of the 
restriction $\leq _ {\times}^0$   of $\leq _ {\times}$  to
 $\bigtimes ^0_{i <  \omega }  \alpha _i$.
However, just considering
 $\bigtimes ^0_{i <  \omega }  2$,
we see that
 $(\bigtimes ^0_{i < \omega }  2 , \leq _ {\times}^0 )$
has linear extensions which are not well-ordered.
Indeed, for every $j < \omega$,
let  $b^j \in \bigtimes ^0_{i < \omega } 2$ 
be defined by $b_j^j = 1$ and  $b_i^j = 0$  if $i \neq j$.
Then the $b^j$'s form a countable set of 
pairwise $\leq^0_{\times}$-incomparable elements
of $\bigtimes ^0_{i < \omega }  2$,
 hence every countable linear order 
is isomorphic to a subset of some linear extension 
of $\bigtimes ^0_{i < \omega }  2$
(e.~g., by \cite[Theorem 3.3]{H}).
Even if we restrict ourselves to well-ordered extensions
of $\bigtimes ^0_{i < \omega }  2$, 
we get from the above considerations that the supremum of
their order-types is $ \omega_1$, hence this supremum 
is not a maximum 
and anyway it is too large to have the intended meaning,
that is, $ \omega = \xx _{i < \omega  } 2 $.

The situation is parallel to \cite{w} and, as in \cite{w},
an order-theoretical characterization can be found
provided we restrict  ourselves to linear extensions satisfying
some finiteness condition.

\subsection{} \labbel{5s} 
We need a bit more  notation 
in order to state the next theorem.
Recall that if  $( \alpha_i) _{i < \omega} $ 
is a sequence of ordinals, then  
$\bigtimes _{i < \omega  }^0  \alpha _i $
is the set of the sequences with finite support
and the (partial) order $\leq_{\times}^0$ is defined componentwise.
If $a,b \in \bigtimes _{i < \omega  }^0  \alpha _i $ 
and $a \neq b$, let $  \diff(a,b)$  be the largest element $i$ of 
$supp(a) \cup supp(b)$ 
 such that 
$a_i \neq b_i$.
Thus (*) above introduces a linear order $< _L$  on
$\bigtimes _{i < \omega  }^0  \alpha _i $ defined by
$a <_L b$ if and only if $a_i < b_i$,
for  $i = \diff(a,b)$ (here $<$ is the standard order
on  $\alpha_i$, hence there is no need to explicitly indicate  the index).
By the results recalled in Subsection \ref{2s}, 
the above  order $< _L$ 
has type  
$\prod _{i < \omega  } \alpha _i $.
On the other hand,
in the finite case, by Carruth theorem mentioned in \ref{3s},
$\bigotimes _{i < n  } \alpha _i $ is the order-type
of the largest linear
extension of $\leq_{\times}$ on $\bigtimes _{i < n  } \alpha _i 
= \bigtimes _{i < n  }^0 \alpha _i$.
We show that in the infinite case
 $\xx _{i < \omega }  \alpha _i $
can be evaluated by combining the above constructions.

If   $a \in \bigtimes _{i <  \omega  } \alpha _i $
is a sequence and $n< \omega$, let 
$a _{\restriction  n} $
 be the restriction of $a$ to $n$,
that is,  $a _{\restriction  n} $ 
is the element of $ \bigtimes _{i <  n} \alpha _i$
defined as follows.
If    $ a= ( a_i) _{i < \omega} $, then  
$a _{\restriction  n} = ( a_i) _{i < n}$. 
Here, as usual, we adopt the convention 
$n= \{ 0, 1, \dots, n-1 \} $. 

Let us say that a linear order $<'$ on  
 $\bigtimes^0 _{i < \omega  } \alpha _i $
is \emph{finitely Carruth}
if  $<'$ extends $< _{\times}^0$ and 
there are an $n < \omega$
and an order 
$\leq ''$ 
extending  $< _{\times}$
on  $\bigtimes _{i < n } \alpha _i $ and such that 
if  $a \neq  b \in \bigtimes^0 _{i <  \omega  } \alpha _i $
and $i= \diff(a,b)$,  then
  \begin{enumerate}
\item
if $i \geq n$, then
 $a <' b$ if and only if 
 $ a_i < b_i $, and
    \item 
if $i <  n$, then 
 $a <' b$ if and only if 
$a _{\restriction  n} <''  a _{\restriction  n} $.
  \end{enumerate}

A linear order $<'$ on  
 $\bigtimes^0 _{i < \omega   } \alpha _i $
is \emph{locally finitely Carruth}
if
  $<'$ extends $< _{\times}^0$ and, 
for every $c \in \bigtimes^0 _{i < \omega  } \alpha _i $,
there is  $n= n_c < \omega$
 such that, for every   $a \neq  b \in \bigtimes^0 _{i <  \omega  } \alpha _i $,
if  $a, b <' c$  and  
$i= \diff(a,b) \geq n$, then
 $a <' b$ if and only if 
 $ a_i < b_i $.
In other words, for every $c$, (1) above holds,
restricted to those pairs of elements 
$a, b $ which are $ <' c$,
while no version of (2) is assumed.

\begin{theorem} \labbel{otcth}
If $( \alpha_i) _{i < \omega} $ 
is a sequence of ordinals, then 
 $ \xx _{i < \omega }  \alpha _i $
is the 
order-type
of some
 finitely Carruth linear order on
$\bigtimes _{i < \omega  }^0  \alpha _i $.

Every 
(locally)
 finitely Carruth linear
order on
$\bigtimes _{i < \omega  }^0  \alpha _i $
is a well-order; moreover,
$ \xx _{i < \omega }  \alpha _i $
is the 
largest order-type
of all such orderings.
 \end{theorem}

\begin{proof}
By  Corollary \ref{segue}, in particular, equation \eqref{2},
there is some  $m< \omega$  
such that 
$ \xx _{ i < \omega } \alpha _i
 =
( \alpha _0 \x \dots \x \alpha _{m-1}) \cdot 
  \prod_{m\leq i < \omega } \alpha _i $.
By Carruth theorem,
there is some 
linear
extension $\leq_C$ of $\leq_{\times}$ on $\bigtimes _{i < m } \alpha _i $ 
such that $P_0 = (\bigtimes _{i < m  } \alpha _i ,\leq_C)$ has
order-type $\bigotimes _{i < m  } \alpha _i $.
By  the results recalled in Subsection \ref{2s}, 
$ \prod_{m\leq i < \omega } \alpha _i $
is the order-type of 
$P_1= (\bigtimes _{m\leq i < \omega } ^0  \alpha_i,  
 \leq_L )$.
Then
 $P= \LL {h < 2} \, P_h$
has order-type 
$( \alpha _0 \x \dots \x \alpha _{m-1}) \cdot 
  \prod_{m\leq i < \omega } \alpha _i 
=  \xx _{ i < \omega } \alpha _i
$, since, as we mentioned, 
for two ordinals $ \gamma _0$ and
$ \gamma _1$  
 the order-type of
 $\LL {i < 2} \, \gamma  _i$
is $ \gamma _0 \gamma  _1$.
Through the canonical bijection
between 
$\bigtimes _{ i < \omega } ^0  \alpha_i$ 
and
$\bigtimes _{i < m } \alpha _i 
 \times \bigtimes _{m\leq i < \omega } ^0  \alpha_i$,
the order $<'$ we have constructed on 
$\LL {i < 2} P _i  $ 
is clearly  
finitely Carruth. Indeed, in the 
definition of finitely Carruth,
take  $n=m$  and
take $\leq''$ as  $\leq_C$.
 If 
$ i= \diff(a,b) < m$,
then 
the ordering between $a$ and $b$
is determined by their $ { \restriction } m $ part, 
since sequences with  the same 
$P_1$-component  are ordered
according to  their $P_0$ component.
Hence
(2) holds.   
On the other hand,
if
$i= \diff(a,b) \geq m$, then    $ a_i < b_i $ if and only if 
$a <' b$, 
by the definition of $\leq_L$, 
thus (1) holds.
Obviously, $<'$ extends $<_{\times}$,
since both 
 $\leq_C$ and $\leq_L$   extend $\leq_{\times}$
on their respective components.  
Hence
$ \xx _{i < \omega }  \alpha _i $
can be realized as the order-type
of some
 finitely Carruth order on
$\bigtimes _{i < \omega  }^0  \alpha _i $
and the first statement is proved.

Since every
 finitely Carruth order on 
$\bigtimes _{i < \omega  }^0  \alpha _i $ is obviously
 locally finitely Carruth,
it is enough to prove that  every locally
 finitely Carruth order on 
$\bigtimes _{i < \omega  }^0  \alpha _i $
is a well-order
of type $\leq \xx _{i < \omega }  \alpha _i$.
So let us assume from now on that 
$<'$ is   locally finitely Carruth.

\begin{claim} \labbel{claim}
If 
$c \in \bigtimes _{i < \omega  }^0  \alpha _i $
 and
$C= \{ a \in  \bigtimes _{i < \omega  }^0  \alpha _i  
\mid a <' c \} $, then 
$(C, <' _{ \restriction C} ) $ is a well-ordered set
of type $\leq \xx _{i < \omega }  \alpha _i$.
 \end{claim}   

\begin{proof}
Let  $n=n_c$  be given by   
local Carruth finiteness and,
as above, let 
$P_1= (\bigtimes _{n\leq i < \omega } ^0  \alpha_i,  
 \leq_L )$.
Notice that, by
 Subsection \ref{2s},
$P_1$ is well-ordered and has  type 
$\prod_{n\leq i < \omega } \alpha _i$. 
If $a \in \bigtimes _{i < \omega  }^0  \alpha _i $,
say,
  $ a= (a_i) _{i < \omega} $, recall that
$a _{\restriction  n} $
is 
the element $  ( a _i) _{i < n}$
 of $ \bigtimes _{i <  n} \alpha _i$.
Similarly, let 
$a _{{\geq} n} $
be
the element $  ( a _i) _{i \geq  n}$
 of $ \bigtimes _{i \geq  n}^0 \alpha _i$.
Thus the position
$ a \mapsto (a _{\restriction  n} , a _{{\geq} n} )$
gives the canonical bijection (mentioned but not described above)
from  
$\bigtimes _{ i < \omega } ^0  \alpha_i$ 
to
$\bigtimes _{i < n} \alpha _i 
 \times \bigtimes _{n \leq i < \omega } ^0  \alpha_i$.
If $P= \{ 
d \in  \bigtimes _{n\leq i < \omega } ^0  \alpha_i
  \mid d=  a _{{\geq} n}
\text{, for some   } a \in 
 \bigtimes _{ i < \omega } ^0  \alpha_i 
\text{ such that  } a <' c \} $,
then $P \subseteq P_1$ hence 
$P$ as a suborder of $P_1$ inherits a well-order of type 
$\leq \prod_{n\leq i < \omega } \alpha _i$.
Moreover, by 
local
 Carruth finiteness,
if
$a,b <' c$ and $a _{{\geq} n}  <_L  b _{{\geq} n} $,
then $a <' b$.   
If $d \in P$, let $Q_d = \{ a _{\restriction  n}  
\mid 
a \in \bigtimes _{i < \omega  }^0  \alpha _i,
 a <' c \text{ and }  a _{{\geq} n} = d \} $. 
Then, for every $d \in P$, the order  $<'$ induces an order 
$ <_d $ on  $Q_d $ by letting
$a _{\restriction  n}  <_d b _{\restriction  n} $ 
if and only if $a <' b$
(notice that, since we are assuming $a,b \in Q_d$,
then $a$ and  $b$ have the same $\geq n$ components).
Since, by assumption, $<'$ extends $<_ {\times}$
on 
$\bigtimes _{ i < \omega } ^0  \alpha_i$,
then 
$ <_d $ extends the restriction of    
$<_ {\times}$ on 
 $\bigtimes _{i < n} \alpha _i $
to $Q_d$.
Hence, by Carruth theorem, 
for every $d \in P$ we have that
$(Q_d, <_d)$ is well-ordered and has type
$\leq \bigotimes _{i < n  } \alpha _i $.

The above considerations show that
$(C, <' _{ \restriction C} ) $ is isomorphic to
the lexicographical product 
$ \text{L} _{d \in P} \, Q_d  $
(recall that if
$a,b <' c$ and $a _{{\geq} n}  <_L  b _{{\geq} n} $,
then $a <' b$).
Since, as we showed,  $P$ is  a well-ordered set  of type 
$\leq \prod_{n\leq i < \omega } \alpha _i$
and each $Q_d$  is a well-ordered set of type
$\leq \bigotimes _{i < n  } \alpha _i $,
then
$(C, <' _{ \restriction C} ) $ is a well-ordered set of order-type 
$\leq \bigotimes _{i < n  } \alpha _i  \cdot \prod_{n\leq i < \omega } \alpha _i$.
 If  
the $m$ given by Corollary \ref{segue} is $\leq n$,
then we immediately get from
equation \eqref{2}
that   $(C, <' _{ \restriction C} ) $ has order-type  $\leq \bigotimes _{i <  \omega  } \alpha _i$.
Otherwise, notice that 
if the condition for 
local
 Carruth finiteness
is satisfied for $c$ and for some $n_c$,
then the condition is satisfied for any 
$n' \geq n_c$ in place of $n_c$,
hence it is no loss of generality to suppose that  the   
$m$ given by Corollary \ref{segue} is $\leq n$
and we are done as before.
\renewcommand{\qedsymbol}{$\Box _{Claim}$}
\end{proof}
 To complete the proof of the theorem,
we have from the Claim that, for every  
$c \in \bigtimes _{i < \omega  }^0  \alpha _i $,
the set of the $<'$-predecessors of $c$ 
is well-ordered; this implies that $<'$
is a well-order. 

It remains to show that $<'$
has order-type $ \leq \xx _{i < \omega }  \alpha _i$. 
This is vacuously true 
if some $\alpha_i$ is $0$ and it follows from Carruth theorem
if the $\alpha_i$'s are eventually $1$. 
Otherwise,  
local
 Carruth finiteness
implies that $<'$ has no maximum. 
Indeed, 
if $a \in \bigtimes _{i < \omega  }^0  \alpha _i $,
then there is $k< \omega$ 
such that $a_i =0$, for $i>k$.
Since the sequence of the $\alpha_i$'s is not eventually $1$
and no $\alpha_i$ is $0$, there is 
$ \bar{\imath}>k$ such that $\alpha_{ \bar{\imath}}>1$.  
If $b$ is equal to $a$ on each component, except that
$b_{ \bar{\imath}}=1$, then
$a < _{\times} b$, hence $a < ' b$,
since, by assumption,  $<'$ extends $< _{\times}$.    
Since $a$ above has been chosen arbitrarily, we get that $<'$
has no maximum.
 Then the Claim implies that
$<'$ has order-type  $\leq \xx _{i < \omega }  \alpha _i$.
\end{proof}

\begin{remark} \labbel{transf}
For the sake of simplicity, we have limited our study here
to sequences of length $ \omega$.
However, essentially all the results here admit a reformulation for the case
of ordinal-indexed transfinite sequences of arbitrary length, modulo the case
of the transfinite natural sums studied in \cite{t}.
It should be remarked that there are different ways
to extend the natural sums and products to sequences of length $> \omega $. 
See \cite[Section 5]{t}, in particular, Definitions 5.2 and Problems 5.6.
We just give here the relevant definitions relative to transfinite products.

Suppose that $( \alpha_ \gamma ) _{ \gamma < \bar{\varepsilon}  } $
is a sequence of ordinals. The
\emph{iterated natural product}
$\prodh _{ \gamma < \delta } \alpha _ \gamma $ is defined
for every $\delta \leq \bar{\varepsilon} $
 as follows.
 \begin{gather*}
\prodh _{ \gamma < 0 }  \alpha _ \gamma  =  1 ;  \quad \quad \quad
\prodh _{ \gamma < \delta+1 } \alpha _ \gamma  =
\left(\prodh _{ \gamma < \delta } \alpha _ \gamma \right)
\x
\alpha_ \delta
\\
 \prodh _{ \gamma < \delta } \alpha _ \gamma  =
\lim _{ \delta ' < \delta }\  \prodh _{ \gamma < \delta' } \alpha _ \gamma
\quad \text{ for $\delta$ limit}  
\end{gather*}
Moreover, we set
\begin{equation*}    
\sideset{}{^o}\xx _{ \gamma < \delta  } \alpha _ \gamma  = 
\inf_ \pi \prodh _{ \gamma < \delta  }  \alpha _{ \pi( \gamma )} 
 \end{equation*} 
where $\pi$ varies among all the permutations of $\delta$.
In the above definition we are keeping $\delta$ fixed.
  In the next
 definition, on the contrary, we let the ordinal $\delta$ vary.
Suppose that $I$ is \emph{any} set and
$( \alpha_i) _{i \in I} $ is a sequence of ordinals.
Define
 \begin{equation*}   
\sideset{}{'}\xx _{ i \in I } \alpha _i = 
\inf_{ \delta , f} \prodh _{ \gamma < \delta  }  \alpha _{ f( \gamma )} 
 \end{equation*} 
where $\delta$ varies among all the ordinals
having cardinality $|I|$ and 
$f$ varies among all the bijections from $\delta$ to $I$.
Furthermore, 
let $\lambda=|I|$ and define
 \begin{equation*}  
\sideset{}{^\bullet}\xx_{ i \in I } \alpha _i = 
\inf_{ f} \prodh _{ \gamma < \lambda  }  \alpha _{ f( \gamma )} 
 \end{equation*} 
where $f$ varies among all the bijections from $ \lambda $ to $I$.
  \end{remark}

\begin{acknowledgement} 
We thank an anonymous referee of \cite{w} for many interesting suggestions
 concerning the relationship between
natural sums and the theory of well-quasi-orders. 
We thank Harry Altman for stimulating discussions.
We thank Arnold W. Miller for letting us know
of Toulmin paper \cite{T} in the MR review of 
\cite{w}. 
 \end{acknowledgement}

The literature on the subject 
of ordinal operations is so vast and sparse 
that we cannot claim completeness of the 
following
list of references.

{\scriptsize
It is not intended that each work in the list
has given equally significant contributions to the discipline.
Henceforth the author disagrees with the use of the list
(even in aggregate forms in combination with similar lists)
in order to determine rankings or other indicators of, e.~g., journals, individuals or
institutions. In particular, the author 
 considers that it is highly  inappropriate, 
and strongly discourages, the use 
(even in partial, preliminary or auxiliary forms)
of indicators extracted from the list in decisions about individuals (especially, job opportunities, career progressions etc.), attributions of funds, and selections or evaluations of research projects.
\par }

\end{document}